\documentclass[twocolumn,8pt]{amsproc}

\usepackage{enumitem}
\usepackage{amsmath,amssymb,amsthm,mathtools,mathrsfs,amsfonts}

\theoremstyle{plain}
\newtheorem{theorem}{Theorem}
\newtheorem{example}[theorem]{Example}
\newtheorem{remark}[theorem]{Remark}
\newtheorem{lemma}[theorem]{Lemma}
\newtheorem{definition}[theorem]{Definition}
\newtheorem{question}[theorem]{Question}

\newtheorem{proposition}[theorem]{Proposition}

\DeclareMathOperator{\codim}{codim}
\DeclareMathOperator{\height}{height}
\DeclareMathOperator{\Spec}{Spec}
\DeclareMathOperator{\Span}{Span}
\DeclareMathOperator{\Ass}{Ass}
\DeclareMathOperator{\grade}{grade}
\DeclareMathOperator{\init}{in}
\DeclareMathOperator{\LT}{LT}
\DeclareMathOperator{\Sp}{S}
\DeclareMathOperator{\lcm}{lcm}

\def\bA{\mathbb{A}}

\begin{document}

\begin{abstract}
We give an ideal in a polynomial ring in three variables, generated by three elements and having height equal to three, but for which no permutation of its generators is a regular sequence. We prove that in a polynomial ring over an infinite field, an ideal of height $n$ generated by $n$ polynomials, always admits a set of $n$ generators, for which any permutation is a regular sequence.
\end{abstract}

\title{Some Remarks on Regular Sequences and Complete Intersections}

\author{Manolis C. Tsakiris}

\thanks{AI tools were not used in the generation of the content of this note.}
\thanks{The author thanks Ethan Reed, a discussion with whom stimulated Question \ref{question}, Aldo Conca for asking a question which became the statement of Theorem \ref{thm:polynomial-ring}, and Alessandro Di Stefani for asking about quasi-regular sequences.}
\thanks{The author is supported by the National Key R$\&$D Program of China (2023YFA1009402).}

\address{State Key Laboratory of Mathematical Sciences \\ Academy of Mathematics and Systems Science \\ Chinese Academy of Sciences \\ 100109, Beijing, manolis@amss.ac.cn}

\maketitle

\section{Introduction}\label{section:Introduction}

\noindent In linear algebra, $n$ equations $$a_{j1} x_1 +\cdots+a_{jd} x_d = b_j, \, \, \, \, j \in [n]$$ in $d$ variables define a solution space of dimension $d-n$, i.e. of codimension equal to the number of equations, if and only if i) they are consistent, and ii) their homogeneous parts 
\begin{align*}
a_{j1} x_1 +\cdots+a_{jd} x_d, \, \, \, j \in [n] 
\end{align*} are linearly independent polynomials. This is a special instance of a more general notion:

\begin{definition}
Let $A$ be a commutative ring and $\underline{\alpha}_n:=\alpha_1,\dots,\alpha_n$ a sequence of elements of $A$. We say $\underline{\alpha}_n$ is a regular sequence of $A$ if i) $A \neq (\underline{\alpha}_n) A$, and ii) for every $j \in [n]$, $\alpha_j$ is a not a zero-divisor of the ring $A/(\alpha_1,\dots,\alpha_{j-1})$.
\end{definition} 

Denote once and for all $I$ the ideal of $A$ generated by the elements $\underline{\alpha}_n$. Let $P$ be any prime ideal minimal over $I$. As regular sequences are preserved by localization \cite[Cor. 1.1.3(a)]{BrunsHerzog}, $\underline{\alpha}_n$ will be a regular sequence of $A_P$. As the length of a maximal regular sequence in a local ring does not exceed the dimension of the ring \cite[Cor. 1.2.13]{BrunsHerzog}, we must have 
$$ n \le \dim(A_P) = \height(P).$$
On the other hand, $P$ is a minimal prime of an ideal generated by $n$ elements, so that Krull's height theorem \cite[Thm. 13.5]{Matsumura} asserts
$ \height(P) \le n.$ Thus every minimal prime of $I$ has height $n$, so 
$\height(I) = n.$ Inasmuch as 
$$\codim \big(\mathcal{V}(I)\big) = \height(I),$$ by definition, where $\mathcal{V}(I)$ is the closed subset of the topological space $\Spec(A)$ consisting of the prime ideals that contain $I$, we see that the ideal generated by a regular sequence of length $n$, defines a locus $\mathcal{V}(I)$ of codimension $n$. 

When $A$ is local Cohen-Macaulay, we have that $\height(I) = n$ if and only if $\underline{\alpha}_n$ is a regular sequence \cite[Thm. 2.1.2(d)]{BrunsHerzog}. If $A$ is Cohen-Macaulay but not necessarily local, an elegant proof has been given in \cite[Thm., Appendix]{Davis} that $\height(I)=n$ implies that $I$ is generated by a regular sequence of length $n$. Still, $\underline{\alpha}_n$ itself may not be a regular sequence:

\begin{example}{\cite[Ex. 1.1.13]{BrunsHerzog}}\label{ex:BrunsHerzog}
Let $A = \mathbb{K}[x,y,z]$ be a polynomial ring over field $\mathbb{K}$ and $$\underline{\alpha}_3 = y(1-x), \, z(1-x), \, x.$$ There is only one minimal prime over $I$,   
$ P = (x,y,z),$ and so $\height(I) = 3$. On the other hand, $z(1-x)$ is a zero-divisor of $A / \big( y(1-x) \big)$. 
\end{example} 

\section{The Question}

To distinguish the phenomenon of Example \ref{ex:BrunsHerzog}, we make the following definition:

\begin{definition}
Let $A$ be a Cohen-Macaulay ring. We call $I = (\alpha_1,\dots,\alpha_n)$ a global complete intersection, if $\height(I) = n$. 
\end{definition}

\begin{remark}
The term ``complete intersection" is traditionally used in algebraic geometry to indicate projective varieties of codimension $n$ whose defining ideal is generated by $n$ homogeneous polynomials, e.g. see \cite[Ex. II.8.4]{Hartshorne} or \cite[11.8]{Harris}. In that case, these polynomials necessarily form a regular sequence, e.g. by \cite[Thm. 1.5.15(c) \& Thm. 2.1.2(d)]{BrunsHerzog}. Commutative algebra texts, such as \cite{Matsumura} or \cite{BrunsHerzog} only define ``complete intersection" in the local case, while in \cite{Eisenbud} the definition of a complete intersection requires $\underline{\alpha}_n$ to be a regular sequence. Our ring-theoretic definition here is more general, and is consistent with \cite[Dfn. 00S9]{stacks-project} and \cite[Dfn. 9.5.8]{Vakil}. In a polynomial ring, it coincides with what Macaulay called in his book ``a module of the principal class" \cite[Dfn. 47]{Macaulay}.
\end{remark}

In a local ring, or more generally in a \emph{$^*$local ring} \cite[Dfn. 1.5.13]{BrunsHerzog}, $\underline{\alpha}_n$ is a regular sequence if and only if $\alpha_{\pi(1)},\dots,\alpha_{\pi(n)}$ is a regular sequence, for any permutation $\pi$ on $[n]=\{1,\dots,n\}$, \cite[Prp. 1.1.6 \& Ex. 1.5.23]{BrunsHerzog}. In the non-local Example \ref{ex:BrunsHerzog}, we saw that $$y(1-x), \, z(1-x), \, x$$ is not a regular sequence, but observe that $$x, \, y(1-x), \, z(1-x)$$ is. One might be tempted to ask:

\begin{question} \label{question}
Let $A$ be a Cohen-Macaulay ring and $I = (\alpha_1,\dots,\alpha_n)$ a global complete intersection. Is it true that some permutation of the $\alpha_j$'s is a regular sequence?
\end{question}

\section{Some Special Cases}

Suppose $n=1$. If $\alpha_1$ were a zero-divisor, then $\alpha_1$ would be contained in some associated prime $P$ of $A$. But as $A$ is Cohen-Macaulay, the ``unmixedness theorem" \cite[Thm. 2.1.6]{BrunsHerzog} applied to the zero ideal, gives that all associated primes of $A$ have height zero. Hence $\height(P) = 0$ and so $\height(\alpha_1)=0$, violating the hypothesis that $\height(\alpha_1) = 1$. Hence the answer to Question \ref{question} is trivially affirmative.

Next, suppose that $A$ is an integral domain and $n=2$. The hypothesis 
\begin{align} \label{n=2}
\height(I) = \height(\alpha_1,\alpha_2) = 2 \tag{$*$}
\end{align} and Krull's height theorem preclude the possibility of either $\alpha_1, \alpha_2$ being zero. As $A$ is an integral domain, $\alpha_1$ is not a zero-divisor, and Krull's Hauptidealsatz gives $\height(\alpha_1) = 1.$ By the unmixedness theorem, $\height(Q) = 1$ for any $Q \in \Spec(A)$ which is an associated prime of $\bar{A}:=A/(\alpha_1)$. Recall the set of zero-divisors of $\bar{A}$ is the union of all such $\bar{Q}$'s, where $\bar{Q}$ indicates the class of $Q$ in $\bar{A}$. Hence if $\bar{\alpha}_2$ were a zero-divisor of $\bar{A}$, we would have $\bar{\alpha}_2 \in \bar{Q}$ for some $Q \in \Ass_A \big(A/(\alpha_1)\big)$. Therefore, the ideal $I$ would be contained in the height-$1$ prime ideal $Q$, ensuing the contradiction $\height(I) = 1$. This again answers Question \ref{question} affirmatively.

\section{The Counterexample} \label{section:Counterexample}

We give a counterexample to Question \ref{question} for $A = \mathbb{K}[x,y,z]$ a polynomial ring in three variables over a field $\mathbb{K}$ of characteristic other than $2$. Consider 
\begin{align*}
\alpha_1 & = x(x+1)yz, \\
\alpha_2 & = x(xy+y+1)(\alpha_1+1), \\
\alpha_3 & = (x+1)(xz+x+2)(\alpha_1 + 1).
\end{align*}

\noindent Let $P$ be minimal prime over $I = (\alpha_1,\alpha_2,\alpha_3)$. Then $P$ must contain at least one factor from each $\alpha_i$. Since $\alpha_1 \in P$, $P$ may not contain the factor $\alpha_1 + 1$ that occurs in $\alpha_2$ and $\alpha_3$. Hence
\begin{enumerate}[leftmargin=*]
\item Either $x \in P$ or $xy+y+1 \in P$.
\item Either $x+1 \in P$ or $xz+x+2 \in P$.
\end{enumerate} Note that it can not be that both $x, x+1$ or $x, xz+x+2$ are in $P$, thus precluding the possibility $x \in P$. Similarly, it can not be that both $x+1,x$ or $x+1,xy+y+1$ are in $P$, precluding the possibility $x+1 \in P$.
In summary, 
\begin{enumerate}[leftmargin=*]
\item $x, \, \, x+1 \not\in P$, 
\item $xy+y+1, \, \, xz+x+2 \in P$.
\end{enumerate} Since not both $y, \, xy+y+1$ can be in $P$, the only factor from $\alpha_1$ that $P$ can contain is $z$, and so $P$ must contain the prime ideal
\begin{align*}
 Q &= (z, \, xy+y+1, \, xz+x+2) \\
& = (x+2, y-1, z).
\end{align*} As $Q$ is a prime ideal, in fact a maximal ideal, the only minimal prime above $I$ is $Q$, therefore $\height(I) = 3$ and $I$ is indeed a global complete intersection. Next, for any $i \neq j$ we have
$$\height(\alpha_i,\alpha_j) = 1,$$
because $\alpha_i$ and $\alpha_j$ share a common factor. But if $\alpha_{\pi(1)}, \alpha_{\pi(2)}, \alpha_{\pi(3)}$ were a regular sequence, then so would be $\alpha_{\pi(1)}, \alpha_{\pi(2)}$, thus forcing 
$$ \height \big(\alpha_{\pi(1)}, \alpha_{\pi(2)} \big)  = 2,$$
by the arguments of \S \ref{section:Introduction}. We conclude that no permutation of the $\alpha_i$'s is a regular sequence. 

We close with the observation that a counterexample for $n=2$ is obtained by taking $\bar{A}  = \mathbb{K}[x,y,z]/(\alpha_1)$ and noting that both $\bar{\alpha}_2, \, \bar{\alpha}_3$ are zero-divisors of $\bar{A}$, but $\height\big(\bar{\alpha}_2, \, \bar{\alpha}_3 \big)=2$. 

\section{A Theorem in Cohen-Macaulay Rings}

In a Cohen-Macaulay ring $A$, we have 
$$ \height(J) = \grade(J),$$
for any ideal $J \neq A$, where $\grade(J)$ indicates the maximal length of a regular sequence of $A$ in $J$ \cite[Cor. 2.1.4]{BrunsHerzog}. Thus, even though, as we have seen, for an ideal $I = (\alpha_1,\dots,\alpha_n)$ of height $n$, no permutation of its generators is guaranteed to be a regular sequence, it is true $\grade(I) = n$ and so $I$ does contain a regular sequence of length $n$. In fact, as mentioned in \S 1, $I$ is generated by a regular sequence of length $n$ \cite[Thm., Appendix]{Davis}. When $A$ has the additional structural feature of being a vector space over an infinite field $\mathbb{K}$, such a regular sequence may be obtained via linear combinations of the generators of $I$ with \emph{generic} coefficients.

To prove a precise statement, we need a technical lemma that concerns the Zariski topology. Recall i) any affine space $\mathbb{K}^d$ is irreducible, ii) any non-empty open set of an irreducible space is dense, iii) any two non-empty open subsets of an irreducible space intersect, and iv) a dense subset intersects any non-empty open set.

\begin{lemma} \label{lem:dense}
Let $\bA = \bA_1 \times \bA_2$ be a product of two affine spaces.
Suppose prescribed dense set $U_1$ of $\bA_1$, and for any $y_1 \in U_1$ dense set $U_{y_1}$ of $\bA_2$. Then the following set is dense in $\bA$:
$$U = \bigcup_{y_1 \in U_1} \{y_1\} \times U_{y_1}$$
\end{lemma}
\begin{proof}
Let $V$ be a non-empty open set of $\bA$. We must prove that $U \cap V$ is non-empty. Let 
$ \pi_1: \bA \rightarrow \bA_1 $ be the canonical projection onto the first factor. If $\pi_1(V)$ were not dense in $\bA_1 $, it would be contained in a closed subset $Y_1 \subsetneq \bA_1$. As $\bA_1 \setminus Y_1$ would be open and non-empty, and $\pi_1$ is surjective and continuous, $\pi_1^{-1} \big(\bA_1 \setminus Y_1 \big)$ would also be open and non-empty. As $\bA$ is irreducible, its two open and non-empty subsets $\pi_1^{-1} \big(\bA_1 \setminus Y_1 \big)$ and $V$ would intersect. But this would imply that $\pi_1(V)$ meets $\bA_1 \setminus Y_1$, which contradicts the definition of $Y_1$. Therefore, $\pi_1(V)$ is dense in $\bA_1$. 

\emph{Chevalley's Theorem} \cite[8.4.2]{Vakil} asserts $\pi_1(V)$ is constructible. As it is dense, it must contain a dense open set $V'_1$ of $\bA_1$. As by hypothesis $U_1$ is dense in $\bA_1$, it must intersect $V'_1$. Take $y_1 \in U_1 \cap V'_1.$ As $y_1 \in \pi_1(V),$ $\pi_1^{-1}(y_1) \cap V$ is a non-empty open set of $\pi_1^{-1}(y_1)$. As $y_1 \in U_1$, by hypothesis $U_{y_1}$ is dense in $\bA_2$, so that $\{y_1\} \times U_{y_1}$ is also dense in $\pi_1^{-1}(y_1)$. As such, $\{y_1\} \times U_{y_1}$ must intersect the non-empty open set $\pi_1^{-1}(y_1) \cap V$. So there is $y_2 \in \bA_2$ such that $$(y_1,y_2) \in  \Big(\{y_1\} \times U_{y_1} \Big) \cap \Big(\pi_1^{-1}(y_1) \cap V\Big).$$ As $(y_1,y_2) \in U \cap V$, the proof is concluded.
\end{proof}  

\begin{theorem} \label{thm:CM}
Let $A$ be a Cohen-Macaulay ring containing an infinite field $\mathbb{K}$, and  let $J = (\beta_1,\dots,\beta_m)$ be an ideal of height $n$. Then there exists a dense set $U \subseteq \mathbb{K}^{n \times m}$, such that for every $C = (c_{ij}) \in U$, and with 
$$b_i := c_{i1} \beta_1 + \cdots + c_{im} \beta_m, \, \, \, i \in [n]$$ the sequence $\underline{b}_n$ is regular in $A$. If in addition $m=n$, then $J = (\underline{b}_n).$
\end{theorem}
\begin{proof}
Consider the $\mathbb{K}$-vector space
$$ V = \Span_{\mathbb{K}}(\beta_1,\dots,\beta_m).$$
Let $P_1,\dots,P_s$ be the associated primes of $A$; these are naturally $\mathbb{K}$-vector subspaces of $A$. Consider the subspaces of $V$ given by
$$ V_j = V \cap P_j.$$
If $V$ were equal to some $V_j$, then all generators of $J$ would be contained in $P_j$, so that $J$ itself would be contained in $P_j$, and $J$ would therefore consist of zero divisors, contradicting $\grade(J) = n>0$. As the field $\mathbb{K}$ is infinite, 
$$Y = \cup_{j \in [s]} V_j \subsetneq V.$$ In fact, $Y$ is a closed subset of $V$, and so $V \setminus Y$ is a dense open set of $V$. Consider the morphism $\mathbb{K}^{1 \times m} \rightarrow V$ that takes an $m$-tuple $(c_{1j})_{j\in [m]}$ to 
$$ c_{11} \beta_1 + \cdots + c_{1m} \beta_m \in V.$$ Let $U_1$ be the inverse image of $V \setminus Y$; it is a dense open subset of $\bA_1$. By construction, it has the property that for any $y_1=(c_{11},\dots,c_{1m}) \in U_1$, 
$$b_1 = c_{11} \beta_1 + \cdots + c_{1m} \beta_m$$ is a non-zero divisor of $A$.

It follows that $\bar{A} = A/(b_1)$ is a Cohen-Macaulay ring. Set $\bar{J} = J+(b_1) / (b_1)$ and note $$\height(\bar{J})=\grade(\bar{J}) = n-1,$$ where the first equality is because $\bar{A}$ is Cohen-Macaulay and the second is due to \cite[Prp. 1.2.10(d)]{BrunsHerzog}. As $\bar{J}$ is generated by the classes $\bar{\beta}_j$ of the $\beta_j$'s in $\bar{A}$, the induction hypothesis settles the existence of a dense set $U_{y_1} \subseteq \mathbb{K}^{(n-1) \times m}$ such that for every $$y_2=(c_{ij})_{i=2,\dots,n; j=1,\dots,m} \in U_{y_1},$$ setting for $i =2,\dots,n$
\begin{align*}
\bar{b}_{i} : = c_{i1} \bar{\beta}_1 + \cdots + c_{im} \bar{\beta}_{m}, 
\end{align*} the sequence $\bar{b}_2,\dots,\bar{b}_n$ of $\bar{A}$ is regular. Set
\begin{align*}
b_{i} &:= c_{i1} \beta_1 + \cdots + c_{im} \beta_{m}, \, \, \, i \in [n] \\
y_1 &:=(c_{1j})_{j=1,\dots,m} \in U_1, \\
y_2 &:=(c_{ij})_{i=2,\dots,n; j=1,\dots,m} \in U_{y_1}.
\end{align*} It follows from the above that $\underline{b}_n$ is a regular sequence in $A$. This is true for any $n \times m$ matrix $$C = (c_{ij}) = (y_1,y_2)$$ that belongs to the set 
$$ U = \bigcup_{y_1 \in U_1} \{y_1\} \times U_{y_1},$$ 
where we view $y_1$ as the first row of the matrix, and $y_2$ as the bottom $(n-1) \times m$ submatrix. By Lemma \ref{lem:dense}, $U$ is a dense subset of $$\mathbb{K}^{n \times m} \cong \mathbb{K}^{1 \times m} \times \mathbb{K}^{(n-1) \times m}.$$

Finally, we argue for the last claim of the Theorem. So suppose $m=n$. Viewing $\underline{b}_n$ and $\underline{\beta}_n$ as column-vectors, we have  
$ \underline{b}_n = C \underline{\beta}_n.$
To show that the $b_j$'s generate $J$, we need to show that $C$ is invertible. If $C$ were not invertible, any left non-zero null-vector $\xi \in \mathbb{K}^{1 \times n}$ of $C$ would induce the linear dependence relation
$$\xi \underline{b}_n = \xi_1 b_1 + \cdots + \xi_n b_n = 0.$$
This would imply that the ideal generated by the $b_j$'s can be generated by $n-1$ elements. In turn, this would imply
$$ \grade( \underline{b}_n)  = \height(\underline{b}_n) \le n-1,$$
where the equality holds because $A$ is Cohen-Macaulay, and the inequality is from Krull's height theorem. This would contradict the fact that $\underline{b}_n$ is a regular sequence of length $n$. 
\end{proof}

\section{A Theorem in Polynomial Rings}

When $A = \mathbb{K}[x_1,\dots,x_m]$ is a polynomial ring over an infinite field we can say even more:

\begin{theorem} \label{thm:polynomial-ring}
Let $A$ be a polynomial ring over an infinite field $\mathbb{K}$, and $I = (\underline{\alpha}_n)$ an ideal of height $n$. Then for $i, j \in [n]$, there exist $c_{ij} \in \mathbb{K}$, such that, with 
$$a_i := c_{i1} \alpha_1 + \cdots + c_{in} \alpha_n, \, \, \, i \in [n]$$
and $\pi$ any permutation on $[n]$, the sequence
$a_{\pi(1)},\dots,a_{\pi(n)}$ is regular. 
\end{theorem}
\begin{proof}
\underline{Step 1.} For $i, j \in [n]$ let $t_{ij}$ be new indeterminates and denote by $\mathbb{K}(t)$ the field of rational functions in the $t_{ij}$'s over $\mathbb{K}$. Set $$A'= \mathbb{K}(t)[x_1,\dots,x_m]$$ and define polynomials in $A'$ as 
$$ g_i = t_{i1} \alpha_1 + \cdots + t_{in} \alpha_n, \, \, \, i \in [n].$$

Fix any term order $\prec$ on $A'$, and with $i$ fixed, apply the Buchberger algorithm \cite[Thm. 2, \S 7, Chp. 2]{CoxLittleOShea} on $g_1,\dots,g_i$ to produce a Gr\"obner basis of the ideal $J = (g_1,\dots,g_i).$ The algorithm terminates with a set of polynomials $G_1,\dots,G_s$ which generate $J$ and whose $S$-polynomials $\Sp(G_\lambda, G_\mu)$ reduce to zero modulo $J$. In particular, there exist polynomials $p_{\mu \nu}$, $q_{u v}$ and $r_{\lambda \mu \nu}$  in $A'$ such that 
\begin{align*}
G_{\mu} &= p_{\mu 1} g_1 + \cdots + p_{\mu i} g_i, \\
g_u & = q_{u1} G_1 + \cdots + q_{us} G_s, \\
\Sp(G_\lambda, G_\mu) & = p_{\lambda \mu  1} g_1 + \cdots + p_{\lambda \mu i} g_i,
\end{align*} for $\mu \in [s]$ and $u \in [i]$. Let $\pi$ be a permutation on $[n]$ and $\varphi$ the ring automorphism of $A'$ that sends $t_{ij}$ to $t_{\pi(i)j}$ and otherwise leaves $\mathbb{K}[x_1,\dots,x_m]$ fixed. Then $\varphi(g_u) = g_{\pi(u)}$ and  
\begin{align*}
\varphi(J) &= \big(g_{\pi(1)},\dots,g_{\pi(i)}\big), \\
\varphi(G_{\mu}) &= \sum\nolimits_{j \in [i]}\varphi(p_{\mu j}) g_{\pi(j)}, \\
g_{\pi(u)} & = \sum\nolimits_{j \in [i]} \varphi(q_{uj}) \varphi(G_j), \\
\varphi\big(\Sp(G_\lambda, G_\mu)\big) & = \sum\nolimits_{j \in [i]} \varphi(p_{\lambda \mu j}) g_{\pi(j)}, \tag{$\#$}
\end{align*} for $\lambda, \mu \in [s]$ and $u \in [i]$. The first three relations express that the $\varphi(G_\mu)$'s generate $\varphi(J)$. 

For a polynomial $F \in A'$, write 
$\LT(F) = c_F \, u_F $ for the leading term of $F$; $u_F$ is the leading monomial of $F$ (it is a monomial in the $x_i$'s) and $c_F \in \mathbb{K}(t)$. As $\varphi$ leaves the $x_i$'s fixed and restricted to $K(t)$ it is an automorphism, 
$$\LT(\varphi(F)) = \varphi(c_F) \, u_F.$$ With $G \in A'$ another polynomial, the pair $F,G$ gives rise to the S-polynomial
$$\Sp(F,G) = \frac{\lcm(u_F,u_G)}{c_F u_F} F - \frac{\lcm(u_F,u_G)}{c_G u_G} G.$$ Immediately one verifies that
$$ \varphi\big(\Sp(F,G)\big) = \Sp\big(\varphi(F), \varphi(G) \big).$$

This, together with relation $(\#)$, implies that the $S$-polynomials of the $\varphi(G_\mu)$'s also reduce to zero modulo $\varphi(J)$, allowing us to conclude via Buchberger's criterion \cite[Thm. 2, \S 7, Chp. 2]{CoxLittleOShea} that the $\varphi(G_\mu)$'s are a Gr\"obner basis for $\varphi(J)$. As the leading terms of a Gr\"obner basis generate the initial ideal, and as 
$u_{G_\mu} = u_{\varphi(G_\mu)},$ we draw the key consequence  
\begin{align}
\init_{\prec}\big(g_1,\dots,g_i\big) = \init_{\prec}\big(g_{\pi(1)},\dots,g_{\pi(i)}\big). \tag{$\ast$} \label{eq:initials-t}
\end{align}

\underline{Step 2.} Consider now the affine space $\bA:= \mathbb{K}^{n \times n}$. For any $n \times n$ matrix $C=(c_{ij}) \in \bA$, and any $i \in [n]$, we obtain polynomial $a_i := g_i|_{C}$ of $A$ by substituting $t_{ij}$ by $c_{ij}$. With $i \in [n]$ fixed and $\pi$ a permutation on $[n]$, we apply the Buchberger algorithm on $a_{\pi(1)},\dots,a_{\pi(i)}$ to obtain a Gr\"obner basis for the ideal $\big(a_{\pi(1)},\dots,a_{\pi(i)}\big)$. Then on a non-empty open subset $V_{\pi,i}$ of $\bA$, the Buchberger algorithm with furnish a Gr\"obner basis for $\big(a_{\pi(1)},\dots,a_{\pi(i)}\big)$ that has the same leading monomials as the Gr\"obner basis for the ideal $\big(g_{\pi(1)},\dots,g_{\pi(i)}\big)$. That is, $\forall C \in V_{\pi,i}$, 
\begin{align*}
&\init_{\prec}\big(a_{\pi(1)},\dots,a_{\pi(i)}\big) \\
&= \init_{\prec}\big(g_{\pi(1)},\dots,g_{\pi(i)}\big) \\
&=\init_{\prec}\big(g_1,\dots,g_i),
\end{align*} the last equality coming from $(\ast)$. As $\bA$ is irreducible, the intersection 
$ V  = \cap_{\pi, i} V_{\pi, i}$ of the finitely many non-empty open sets $V_{\pi, i}$ is also non-empty and open. By construction, for any $C \in V$, the $a_j$'s have the property that the initial ideal of the ideal generated by any $i$ of them, coincides with $\init_{\prec}\big(g_1,\dots,g_i\big)$. 

\underline{Step 3.} By Theorem \ref{thm:CM} there exists a dense subset $U \subseteq \bA$, such that for any $C=(c_{ij}) \in U$ the polynomials $a_1,\dots,a_n$ are a regular sequence of $A$. As $U$ is dense, it intersects the non-empty open set $V$ of Step 2. Taking $C \in U \cap V$, one traces a chain of equalities
\begin{align*}
&\grade \big(a_{\pi(1)},\dots,a_{\pi(i)}\big) \\
&=\height\big(a_{\pi(1)},\dots,a_{\pi(i)}\big) \\
& = \height \Big( \init_{\prec} \big(a_{\pi(1)},\dots,a_{\pi(i)}\big) \Big) \\
& = \height \Big( \init_{\prec} \big(g_{\pi(1)},\dots,g_{\pi(i)}\big) \Big) \\
&= \height \big(\init_{\prec}(g_1,\dots,g_i) \big) \\
& = \height \big( \init_{\prec}(a_1,\dots,a_i)\big)  \\
& =  \height (a_1,\dots,a_i)  \\
& = \grade(a_1,\dots,a_i) = i,
\end{align*} with equalities 2 and 6 because the height of an ideal is always equal to the height of its initial ideal \cite[Proposition 1.6.2(a)]{AldoBook}, equality 3 because $C \in V_{\pi,i}$, equality 4 being \eqref{eq:initials-t}, equality 5 because $C \in V_{\text{id},i}$,  equalities 1 and 7 because $A$ is Cohen-Macaulay, and equality 8 because $C \in U$ whence $\underline{a}_i$ is a regular sequence. 

\underline{Step 4.} We now argue that for $C$ as in Step 3, $a_{\pi(1)},\dots,a_{\pi(n)}$ is a regular sequence in $A$ for any $\pi$. We argue inductively on $i$ that $$a_{\pi(1)},\dots,a_{\pi(i)}$$ is a regular sequence in $A$. This is true for $i=1$, because each $a_j$ is a non-zero element of the integral domain $A$. For $i>1$, the induction hypothesis asserts $a_{\pi(1)},\dots,a_{\pi(i-1)}$ is a regular sequence. It suffices to prove $a_{\pi(i)}$ is not a zero-divisor of $A / \big(a_{\pi(1)},\dots,a_{\pi(i-1)}\big).$ By Step 3
$$\grade \big(a_{\pi(1)},\dots,a_{\pi(i)}\big) = i.$$ By the \emph{depth densitivity} of the Koszul complex \cite[Thm. 16.8]{Matsumura},  
$$H_1\big(a_{\pi(1)},\dots,a_{\pi(i)}\big) = 0,$$ i.e. the homology of the Koszul complex associated to the elements $a_{\pi(1)},\dots,a_{\pi(i)}$ at homological position $1$ vanishes. This is tantamount to saying that every $i$-tuple of polynomials $(p_1,\dots, p_i)$ such that  
\begin{align}
p_1 a_{\pi(1)} + \cdots + p_i a_{\pi(i)} = 0, \tag{$\lozenge$} \label{eq:syzygy}
\end{align}
can be expressed as an $A$-linear combination of \emph{Koszul syzygies}, i.e. of $i$-tuples $\mathfrak{K}_{k,\ell}$, each associated with a pair $k < \ell$ of integers in $[i]$, for which the $k$th entry is $-a_{\pi(\ell)}$, the $\ell$th entry is $a_{\pi(k)}$, and all other entries are zero. Precisely, there are polynomials $f_{k,\ell} \in A$ such that 
$$(p_1,\dots, p_i) = \sum_{1 \le k < \ell \le i} f_{k,\ell} \, \mathfrak{K}_{k,\ell}.$$ From the above relation we read 
\begin{align}
p_i = \sum_{k=1}^{i-1} f_{k,i} \, a_{\pi(k)}. \tag{$\lozenge\lozenge$} \label{eq:regular}
\end{align} Relations \eqref{eq:syzygy} and \eqref{eq:regular} express that $\alpha_{\pi(i)}$ is not a zero-divisor modulo $\big(a_{\pi(1)},\dots,a_{\pi(i-1)}\big)$. 
\end{proof}

\section{Quasi-Regular Sequences}

Another classical notion, weaker than the notion of a regular sequence, is that of a \emph{quasi-regular} sequence \cite[Dfn. 15.1.7, Ch. 0]{Grothendieck}.

\begin{definition}
We say $\underline{\alpha}_n$ is a quasi-regular sequence, if $I=(\underline{\alpha}_n) \neq A$ and the canonical $A/I$-algebra map $$A/I[x_1,\dots,x_n] \rightarrow \oplus_{d \ge 0} I^d/I^{d+1}$$ that takes $x_i$ to the class of $\alpha_i$, is bijective. 
\end{definition}

In a local Noetherian ring the two notions coincide, while in general one only has that a regular sequence is quasi-regular. Note that by definition, $\underline{\alpha}_n$ is quasi-regular if and only if any permutation of $\underline{\alpha}_n$ is. We have:

\begin{proposition}
If $A$ is Cohen-Macaulay and $\height(I)=n$, then $\underline{\alpha}_n$ is quasi-regular. 
\end{proposition}
\begin{proof}
As $\grade(I) = n$, \cite[Thm. 16.8]{Matsumura} asserts that the first Koszul homology associated to $\underline{\alpha}_n$ vanishes, i.e. $H_1(\underline{\alpha}_n)=0$. Then assertion (iii)$\Rightarrow$(v) of 
\cite[Th\'eor\`em 1, \S 9, Ch. X]{Bourbaki} implies that 
$\underline{\alpha}_n$ is quasi-regular (see also \cite[Rem. 1.6.20]{BrunsHerzog}).
\end{proof}

\noindent In particular, sequence $\underline{\alpha}_3$ in \S \ref{section:Counterexample} is quasi-regular.

\bibliographystyle{alpha}

\begin{thebibliography}{{{Sta}}26}

\bibitem[BCRV22]{AldoBook}
Winfried Bruns, Aldo Conca, Claudiu Raicu, and Matteo Varbaro.
\newblock {\em {Determinants, Gr{\"o}bner Bases and Cohomology}}.
\newblock Springer, 2022.

\bibitem[BH98]{BrunsHerzog}
Winfried Bruns and Jurgen Herzog.
\newblock {\em {Cohen-Macaulay Rings}}.
\newblock Cambridge University Press, 1998.

\bibitem[Bou80]{Bourbaki}
Bourbaki.
\newblock {\em {Alg\'ebre, Chapitre 10}}.
\newblock Springer, 1980.

\bibitem[CLO15]{CoxLittleOShea}
David~A. Cox, John Little, and Donal O'Shea.
\newblock {\em {Ideals, Varieties,and Algorithms: An Introduction to
  Computational Algebraic Geometry and Commutative Algebra}}.
\newblock Springer, 2015.

\bibitem[Dav67]{Davis}
Edward~D. Davis.
\newblock {{Ideals of the Principal Class, R-Sequences and a Certain Monoidal
  Transformation}}.
\newblock {\em Pacific Journal of Mathematics}, 20(2), 1967.

\bibitem[Eis94]{Eisenbud}
David Eisenbud.
\newblock {\em {Commutative Algebra: with a View Toward Algebraic Geometry}}.
\newblock Springer, 1994.

\bibitem[GD64]{Grothendieck}
A.~Grothendieck and J.~Dieudonn{\'e}.
\newblock {\em {\'E}l{\'e}ments de g{\'e}om{\'e}trie alg{\'e}brique: {IV}.
  {{\'E}}tude locale des sch{\'e}mas et des morphismes de sch{\'e}mas,
  {Premi{\`e}re} partie}.
\newblock Institut des Hautes {\'E}tudes Scientifiques, 1964.

\bibitem[Har77]{Hartshorne}
Robin Hartshorne.
\newblock {\em {Algebraic Geometry}}.
\newblock Springer, 1977.

\bibitem[Har92]{Harris}
Joe Harris.
\newblock {\em {Algebraic Geometry: A First Course}}.
\newblock Springer, 1992.

\bibitem[Mac16]{Macaulay}
Francis~S. Macaulay.
\newblock {\em {The Algebraic Theory of Modular Systems}}.
\newblock Cambridge University Press, 1916.

\bibitem[Mat89]{Matsumura}
Hideyuki Matsumura.
\newblock {\em {Commutative Ring Theory}}.
\newblock Cambridge University Press, 1989.

\bibitem[{{Sta}}26]{stacks-project}
The {{Stacks Project Authors}}.
\newblock {The Stacks Project}.
\newblock https://stacks.math.columbia.edu, 2026.

\bibitem[Vak24]{Vakil}
Ravi Vakil.
\newblock {\em {Foundations of Algebraic Geometry}}.
\newblock 2024.

\end{thebibliography}

\end{document}